\documentclass[oneside,english]{amsart}
\usepackage[T3,T1]{fontenc}
\usepackage[latin9]{inputenc}
\usepackage{xcolor}
\usepackage{babel}
\usepackage{mathtools}
\usepackage{amstext}
\usepackage{amsthm}
\usepackage{amssymb}
\usepackage{cancel}
\usepackage{mathdots}
\usepackage[bookmarks=true,bookmarksnumbered=true,bookmarksopen=false,
 breaklinks=true,pdfborder={0 0 1},backref=false,colorlinks=true]
 {hyperref}
\hypersetup{
 linkcolor=blue,citecolor=blue}

\makeatletter
\numberwithin{equation}{section}
\numberwithin{figure}{section}

\PassOptionsToPackage{english}{babel}
\usepackage{babel}
\usepackage{enumerate}
\usepackage{verbatim}
\usepackage{mathrsfs}
\usepackage{cite}
\usepackage{bbm}

\mathtoolsset{showonlyrefs}

\usepackage{physics}
\usepackage{tikz}
\usepackage{yhmath}
\usepackage{siunitx}
\usepackage{array}
\usepackage{multirow}
\usepackage{gensymb}
\usepackage{tabularx}
\usepackage{extarrows}
\usepackage{booktabs}
\usetikzlibrary{fadings}
\usetikzlibrary{patterns}
\usetikzlibrary{shadows.blur}
\usetikzlibrary{shapes}
\usetikzlibrary{arrows.meta}

\tikzset{
    ShortArrow/.tip={Triangle[length=2pt, width=3pt]},
    >=ShortArrow
}

\providecommand{\keywords}[1]{\textbf{Index terms---} #1}

\DeclareSymbolFont{tipa}{T3}{cmr}{m}{n}
\DeclareMathAccent{\inbreve}{\mathalpha}{tipa}{16}

\newcommand{\subsetsim}{\mathrel{%
  \ooalign{\raise0.2ex\hbox{$\subset$}\cr\hidewidth\raise-0.8ex\hbox{\scalebox{0.9}{$\sim$}}\hidewidth\cr}}}
\newcommand{\supsetsim}{\mathrel{%
  \ooalign{\raise0.2ex\hbox{$\supset$}\cr\hidewidth\raise-0.8ex\hbox{\scalebox{0.9}{$\sim$}}\hidewidth\cr}}}

\newcommand{\subsetapprox}{\mathrel{%
  \ooalign{\raise0.4ex\hbox{$\subset$}\cr\hidewidth\raise-0.8ex\hbox{\scalebox{0.9}{$\approx$}}\hidewidth\cr}}}
\allowdisplaybreaks[1]
\def\M{\mathcal{M}}
\def\X{\mathcal{X}}
\def\P{\mathscr{P}}

\def\1{\mathbf{1}}

\def\d{{\text {\rm d}}}

\theoremstyle{definition}
\theoremstyle{definition}
\newtheorem{condition}{Condition}

\newtheorem{theorem}{Theorem}[section]
\newtheorem{lemma}{Lemma}[section]

\providecommand{\remarkname}{Remark}
\providecommand{\theoremname}{Theorem}

\providecommand{\definitionname}{Definition}

\providecommand{\examplename}{Example}

\makeatother

\theoremstyle{plain}
\newtheorem{thm}{\protect\theoremname}
\theoremstyle{definition}
\newtheorem{defn}[thm]{\protect\definitionname}
\theoremstyle{remark}
\newtheorem{rem}[thm]{\protect\remarkname}
\theoremstyle{definition}
\newtheorem{example}[thm]{\protect\examplename}
\providecommand{\definitionname}{Definition}
\providecommand{\examplename}{Example}
\providecommand{\remarkname}{Remark}
\providecommand{\theoremname}{Theorem}

\begin{document}
\title{Sharp Small-Volume Isoperimetry from Log-Sobolev Inequalities}
\author{Lei Yu }
\address{L. Yu is with the School of Statistics and Data Science, LPMC, KLMDASR,
and LEBPS, Nankai University, Tianjin 300071, China (e-mail: leiyu@nankai.edu.cn).}
\begin{abstract}
Let $I_{\inf}(a):=\lim_{n\to\infty}I_{n}(a)=\inf_{n\ge1}I_{n}(a)$
be the asymptotic isoperimetric profile for product of a weighted
Riemannian manifold satisfying $\mathrm{CD}(0,\infty)$ and, more
broadly, for product of a nonsmooth subspace with density that can
be approximated by a sequence of densities satisfying $\mathrm{CD}(0,\infty)$.
We establish the following identity: 
\[
\lim_{a\downarrow0}\frac{I_{\inf}(a)}{a\sqrt{2\ln(1/a)}}=\sqrt{K_{\mathrm{LS}}},
\]
where $K_{\mathrm{LS}}$ is the optimal log-Sobolev constant. The
same argument also yields the large-deviation and moderate-deviation
asymptotics for isoperimetry. 
\end{abstract}

\maketitle

\section{Introduction}

We consider a smooth manifold $\M$ satisfying the following regular
condition. 

\begin{condition}[Smoothness]\label{cond:smoothness} We assume
that $(\M,g)$ be a smooth, complete, oriented, connected Riemannian
manifold of dimension $k\ge1$, with geodesic distance $d$ induced
by $g$; we adopt $(\M,g)=(\mathbb{R},|\cdot|)$ for $k=1$. We assume
that $\nu=e^{-V}\mathrm{vol}$ is an absolutely continuous Borel probability
measure on $\M$ with $V\in C^{2}(\M)$. \end{condition}

A natural convexity assumption on a smooth manifold with density is
the curvature-dimension condition $\mathrm{CD}(0,\infty)$. 

\begin{condition}[Convexity]\label{cond:convexity} We assume that
$(\M,g,\nu)$ satisfies nonnegative Bakry--Émery--Ricci curvature
(i.e., $\mathrm{CD}(0,\infty)$):
\begin{equation}
\mathrm{Ric}_{V}:=\mathrm{Ric}_{g}+\mathrm{Hess}_{g}V\geq0\;\;\text{ as 2-tensor fields},\label{eq:RicV}
\end{equation}
where $\mathrm{Ric}_{g}$ is the Ricci curvature and $\mathrm{Hess}_{g}V$
is the Hessian of $V$. \end{condition}

We consider the $n$-fold product space $\M^{n}$, equipped with the
product measure $\nu_{n}:=\nu^{\otimes n}$ and the product distance
\[
d_{n}(\mathbf{x},\mathbf{y})=\sqrt{\sum^{n}_{i=1}d(x_{i},y_{i})^{2}},\quad\mathbf{x}=(x_{1},\dots,x_{n}).
\]

For a set $A\subseteq\M^{n}$, its closed $r$-enlargement is defined
as 
\begin{equation}
A^{r}:=\bigcup_{\mathbf{x}\in A}\{\mathbf{y}\in\M^{n}:d_{n}(\mathbf{x},\mathbf{y})\leq r\}.\label{eq:Gamma-1}
\end{equation}
For any Borel set $A\subseteq\M^{n}$, its perimeter is given by 
\begin{equation}
\nu^{+}_{n}(A):=\liminf_{r\downarrow0}\frac{\nu_{n}(A^{r})-\nu_{n}(A)}{r}.\label{eq:perimeter}
\end{equation}
The isoperimetric profile is then the minimal perimeter over all sets
of measure $a\in[0,1]$: 
\begin{equation}
I_{n}(a):=\inf\bigl\{\nu^{+}_{n}(A):A\subseteq\M^{n}\text{ Borel},\ \nu_{n}(A)=a\bigr\}.\label{eq:Gamma}
\end{equation}
We define the asymptotic isoperimetric profile 
\[
I_{\inf}(a):=\lim_{n\to\infty}I_{n}(a)=\inf_{n\ge1}I_{n}(a).
\]

 A related notion is the isoperimetric profile for the infinite dimensional
product space $\M^{\infty}$ equipped with the product topology and
the product measure $\nu_{\infty}:=\nu^{\otimes\infty}$, which is
defined by 
\begin{equation}
I_{\infty}(a):=\inf_{\mathrm{closed}\,A:\nu_{\infty}(A)=a}\nu^{+}_{\infty}(A),\;a\in[0,1].\label{eq:I_infty}
\end{equation}
Here, the infimum is taken over all sets $A\subseteq\M^{\infty}$
that are closed under the product topology and have $\nu_{\infty}$-measure
$a$, and in the definition of $\nu^{+}_{\infty}(A)$, the enlargement
is defined under $d_{\infty}(\mathbf{x},\mathbf{y})=\sqrt{\sum^{\infty}_{i=1}d(x_{i},y_{i})^{2}}$
(although it might diverse to infinity for two arbitrary points).
The author showed the identity of $I_{\inf}$ and $I_{\infty}$ (under
the axiom of countable choice) \cite{yu2025large}.

\begin{defn}
The small-volume isoperimetric constant is defined as 
\[
K^{-}_{\mathrm{IS}}:=\liminf_{a\downarrow0}\frac{I_{\inf}(a)^{2}}{2a^{2}\ln(1/a)}.
\]
\end{defn}

This paper also involves the log-Sobolev inequality and its nonlinear
version. 
\begin{defn}
\label{def:nonlinearLS} For a convex function $\theta:[0,\infty]\to[0,\infty]$
(satisfying $\theta(0)=0$), we say a probability measure $\nu$ admits
the \emph{$\theta$-nonlinear log-Sobolev inequality} (shortly, \emph{$\theta$-nonlinear
LSI}) if 
\begin{equation}
\theta(D(\mu\|\nu))\leq I(\mu\|\nu),\;\forall\mu\in\P^{\mathrm{ac}}(\M),\label{eq:-74}
\end{equation}
where for $\mu=\rho\nu$, the relative entropy is
\begin{equation}
D(\mu\|\nu):=\int\rho\ln\rho\,\d\nu\label{eq:entropy-1}
\end{equation}
and the Fisher information is 
\begin{equation}
I(\mu\|\nu):=\begin{cases}
4\int|\nabla\sqrt{\rho}|^{2}_{*}\d\nu=\int_{\{\rho>0\}}\frac{|\nabla\rho|^{2}_{*}}{\rho}\d\nu, & \text{if }\sqrt{\rho}\in W^{1,2}(\M,d,\nu)\\
+\infty, & \text{otherwise}
\end{cases}\label{eq:Fisher}
\end{equation}
with $|\nabla f|_{*}$ denoting the minimal relaxed gradient of
a function $f$ \cite{ambrosio2014calculus,gigli2013log}. Define
the \emph{log-Sobolev profile} as 
\[
\Theta(\alpha):=\inf_{\mu\in\P^{\mathrm{ac}}(\M):D(\mu\|\nu)=\alpha}I(\mu\|\nu),
\]
and its lower convex envelope, identified with the optimal convex
function $\theta$ for the nonlinear LSI in \eqref{eq:-74}, can be
written as 
\begin{align}
\breve{\Theta}(\alpha) & =\inf_{\substack{\lambda\in[0,1],\mu_{1},\mu_{2}\in\P^{\mathrm{ac}}(\M):\\
\lambda D(\mu_{1}\|\nu)+(1-\lambda)D(\mu_{2}\|\nu)=\alpha
}
}\lambda I(\mu_{1}\|\nu)+(1-\lambda)I(\mu_{2}\|\nu).\label{eq:Theta_LCE}
\end{align}
\end{defn}

The definition of nonlinear LSI is still valid for Polish metric probability
measure spaces. We may restrict the probability measures $\mu$ in
the nonlinear LSI to some more regular classes. 
\begin{itemize}
\item Given any Borel probability measure $\nu$ on a Polish metric space
$(\X,d)$, when $\theta$ is continuous, we may assume the probability
measures $\mu$ admit Lipschitz densities  $\rho=\d\mu/\d\nu:\X\to[0,\infty)$,
and for this case, $|\nabla\rho|_{*}$ can be replaced by the local
Lipschitz constant $|\nabla\rho|$ \cite{gigli2013log}. 
\item If $\nu$ satisfies the smoothness assumption in Condition \ref{cond:smoothness},
 by approximation arguments, we may assume that the probability measures
$\mu$ have smooth density $\rho:=\d\mu/\d\nu$, in which case, $|\nabla\rho|_{*}=|\nabla\rho|$
is the norm of the gradient $\nabla\rho$. 
\end{itemize}
For a convex $\theta$, a $\theta$-nonlinear LSI holds only if $\theta(0)=0$
and $\theta$ is non-decreasing, and thus, $\breve{\Theta}(0)=0$
and $\breve{\Theta}$ is non-decreasing and convex. In particular,
when $\theta(t)=2Kt$ for some constant $K>0$, then the $\theta$-nonlinear
LSI reduces to the standard LSI. Let $K_{\mathrm{LS}}$ be the optimal
log-Sobolev constant. 

The author \cite{yu2025large} showed that $K^{-}_{\mathrm{IS}}\le K_{\mathrm{LS}}$
and asked whether $K^{-}_{\mathrm{IS}}=K_{\mathrm{LS}}.$ We answer
this question in the present paper. 

We first establish a bridge between isoperimetric and log-Sobolev
inequalities in finite-dimensional spaces. 

\begin{theorem} \label{thm:finite-dimensional} Let $(\M,d,\nu)$
be a weighted Riemannian manifold satisfying Conditions \ref{cond:smoothness}
and \ref{cond:convexity}. Then, it holds that for every $\delta>0$,
every $n\ge1$, and every $a\in(0,1)$, 
\begin{equation}
I_{n}(a)\ge\frac{a}{1+2\delta}\sqrt{n\breve{\Theta}\left(\frac{1}{n}(\ln\frac{1}{a}-C_{\delta})_{+}\right)}.\label{eq:theta}
\end{equation}
In particular, 
\begin{equation}
I_{n}(a)\ge\frac{a}{1+2\delta}\sqrt{2K_{\mathrm{LS}}\left(\ln\frac{1}{a}-C_{\delta}\right)_{+}}.\label{eq:linear}
\end{equation}
\end{theorem}

By using this theorem and Theorem 1 of \cite{yu2025large}, we prove
$K^{-}_{\mathrm{IS}}=K_{\mathrm{LS}}.$

\begin{theorem}[Small-Volume Isoperimetry] \label{thm:main} For
a weighted Riemannian manifold $(\M,d,\nu)$ satisfying Conditions
\ref{cond:smoothness} and \ref{cond:convexity}, it holds that 
\[
K^{-}_{\mathrm{IS}}=K_{\mathrm{LS}}.
\]
More precisely, 
\begin{equation}
\lim_{a\downarrow0}\frac{I_{\inf}(a)}{a\sqrt{2\ln(1/a)}}=\sqrt{K_{\mathrm{LS}}}.\label{eq:CLT}
\end{equation}
\end{theorem}

Since $I_{\inf}=I_{\infty}$ (under the axiom of countable choice)
\cite{yu2025large}, $I_{\inf}$ is exactly the isoperimetric profile
for the infinite dimensional product space $\M^{\infty}$. The equivalence
in \eqref{eq:CLT} characterizes the asymptotic behaviour of this
isoperimetric profile as $a\downarrow0$ (or $a\uparrow1$).

The equivalence in \eqref{eq:CLT} implies Ledoux and Bobkov's weak
equivalence \cite{ledoux1994semigroup,ledoux1994simple,bobkov1999isoperimetric}.
That is, a Gaussian-type isoperimetric inequality $I_{1}(a)\ge\sqrt{K}I_{\mathrm{G}}(a),\forall a$
with $I_{\mathrm{G}}$ denoting Gaussian isoperimetric profile holds
for some $K>0$, if and only if the standard LSI holds with some constant
$K'>0$. On one hand, by the facts that $I_{1}\ge I_{\inf}$, if the
standard LSI holds with some constant $K'>0$, then our equivalence
implies $\liminf_{a\to0}\frac{I_{1}(a)}{I_{\mathrm{G}}(a)}\ge\liminf_{a\to0}\frac{I_{\inf}(a)}{I_{\mathrm{G}}(a)}=\sqrt{K'}$.
That is, $I_{1}(a)\ge\sqrt{K}I_{\mathrm{G}}(a),\forall a$ for some
$K>0$, yielding a Gaussian-type isoperimetric inequality for the
space $\M$. On the other hand, if a Gaussian-type isoperimetric inequality
for $\M$ holds with constant $K>0$, then by the well known fact
that $I_{1}(a)\ge\sqrt{K}I_{\mathrm{G}}(a),\forall a\Longleftrightarrow I_{n}(a)\ge\sqrt{K}I_{\mathrm{G}}(a),\forall a,n$,
our equivalence implies the LSI with the same constant $K$.

Our second result is establishing the moderate-deviation principle
for isoperimetry. 

\begin{theorem}[Moderate-Deviation Principle for Isoperimetry] \label{thm:MDP}
Let $(\M,d,\nu)$ be a weighted Riemannian manifold satisfying Conditions
\ref{cond:smoothness} and \ref{cond:convexity}. Let $a_{n}=e^{-s_{n}},$
where $s_{n}\to\infty,\,s_{n}/n\to0.$ Then,
\begin{equation}
\lim_{n\to\infty}\frac{I_{n}(a_{n})}{a_{n}\sqrt{2s_{n}}}=\sqrt{K_{\mathrm{LS}}}.\label{eq:MDP}
\end{equation}
\end{theorem}

The same argument is valid for deriving the large-deviation principle
for isoperimetry. 

\begin{theorem}[Large-Deviation Principle for Isoperimetry] \label{thm:LDP}
Let $(\M,d,\nu)$ be a weighted Riemannian manifold satisfying Conditions
\ref{cond:smoothness} and \ref{cond:convexity}. Given $\alpha\ge0$,
let $a_{n}=e^{-n\alpha}$. Then,
\begin{equation}
\lim_{n\to\infty}\frac{I_{n}(a_{n})}{a_{n}\sqrt{n}}=\sqrt{\breve{\Theta}(\alpha)}.\label{eq:LDP}
\end{equation}
\end{theorem}

Large-deviation principle for isoperimetry was first established
by the author in \cite{yu2025large} but under the condition that
curvature is bounded below when $\nu$ is a Gaussian-like probability
measure. Theorem \ref{thm:LDP} removes this limitation. Moreover,
the proof of Theorem \ref{thm:LDP} is much simpler than the one in
\cite{yu2025large}.

Although the theorems above are valid for spaces that are general
enough, they do not cover isoperimetric problems for log-concave probability
measures defined on closed convex subsets of Euclidean spaces. To
cover them, we extend the theorem above to nonsmooth subspaces of
the Riemannian manifold $\M$. 

\begin{theorem}[Nonsmooth Subspaces] \label{thm:approximate} Let
$\nu$ be an absolutely continuous probability measure with density
$f$ (w.r.t. Riemannian volume) on a finite-dimensional smooth complete
oriented connected Riemannian manifold $\M$. Moreover, suppose that
there exists a sequence of probability measures $\nu_{m},m\in\mathbb{N}$
with densities $f_{m}$ (w.r.t. Riemannian volume) satisfying Conditions
\ref{cond:smoothness} and \ref{cond:convexity} and $f_{m}\ge(1-1/m)f$.
Then, \eqref{eq:CLT}, \eqref{eq:MDP}, and \eqref{eq:LDP} still
hold for $(\M,d,\nu)$. \end{theorem}
\begin{rem}
The convergence of $\{\nu_{m}\}$ here is called ``convergence from
above'' by E. Milman \cite{milman2009role}. In addition, another
notion, ``convergence from within'', was also introduced by him.
Theorem \ref{thm:approximate} still holds if convergence of $\{\nu_{m}\}$
from above is replaced by convergence from within.
\end{rem}

\begin{example}[Log-Concave Probability Measures]
\label{exa:Hypercube} All smooth log-concave probability densities
fully supported on Euclidean spaces satisfy Conditions \ref{cond:smoothness}
and \ref{cond:convexity}. Moreover, by standard regularization techniques
(e.g., the Moreau-Yosida regularization and convolution with Gaussian
kernels), all log-concave probability densities defined on closed
convex subsets of Euclidean spaces can be approximated by smooth log-concave
probability densities defined on the whole Euclidean spaces in the
sense as in Theorem \ref{thm:approximate}. Thus, \eqref{eq:CLT},
\eqref{eq:MDP}, and \eqref{eq:LDP} hold for all (smooth or non-smooth)
log-concave probability measures on Euclidean spaces (supported on
the whole spaces or convex subsets). The standard Gaussian measure
and the uniform measure on the hypercube $[0,1]^{n}$ are two of the
most important examples; $K_{\mathrm{LS}}=1$ for the former and $K_{\mathrm{LS}}=\pi^{2}$
for the latter. 
\end{example}

\section{Proof of Theorem \ref{thm:finite-dimensional}}

A key tool is the following lemma.

\begin{lemma}[Log-Sobolev-Concentration Lemma] \label{lem:LSC} Let
$(\X,d,P)$ be a metric probability space on which the Fisher information
is defined in the usual Sobolev sense (see e.g. \cite{yu2025large}).
Let $A\subseteq\X$ be closed with $P(A)=a\in(0,1).$ Assume that
for some $h>0$, 
\begin{equation}
P(A^{r})\le ae^{hr},\quad\forall r\ge0,\label{eq:growth}
\end{equation}
where $A^{r}:=\{x\in\X:d(x,A)\le r\}.$ Then, for every $\delta>0$,
there exists a probability measure $Q_{\delta}\ll P$ such that 
\begin{align}
D(Q_{\delta}\Vert P) & \ge(\ln\frac{1}{a}-C_{\delta})_{+},\label{eq:entropy-l}\\
I(Q_{\delta}\Vert P) & \le(1+2\delta)^{2}h^{2},\label{eq:fisher-u}
\end{align}
where 
\begin{equation}
C_{\delta}:=\frac{1+2\delta}{\delta}\ln\frac{1+\delta}{\delta}.\label{eq:Cdelta}
\end{equation}
Here $h$ can be chosen as the best one $h=\sup_{r>0}\frac{1}{r}\ln\frac{P(A^{r})}{a}$.

\end{lemma}
\begin{proof}[Proof of Lemma \ref{lem:LSC}]
Let $R(x):=d(x,A)$ and, for $u>1$, define 
\[
Z(u):=\int_{\X}e^{-uhR(x)}\,\d P(x).
\]

Since $R=0$ on $A$, it holds that $Z(u)\ge a.$ 

We also obtain an upper bound. By layer-cake representation, we obtain
that 
\begin{align}
Z(u) & =uh\int^{\infty}_{0}e^{-uhr}P(R\le r)\,\d r.
\end{align}
By \eqref{eq:growth}, 
\begin{align}
Z(u)\le auh\int^{\infty}_{0}e^{-(u-1)hr}\,\d r & =a\frac{u}{u-1}.\label{eq:Z-upper}
\end{align}
Hence 
\begin{equation}
a\le Z(u)\le a\frac{u}{u-1}.\label{eq:Z-two-sided}
\end{equation}

Now fix $\delta>0$ and set $s:=1+2\delta,\,t:=1+\delta.$ Define
a probability measure $Q_{\delta}$ by 
\begin{equation}
\frac{\d Q_{\delta}}{\d P}=\frac{e^{-shR}}{Z(s)}.\label{eq:q-def}
\end{equation}

We first estimate its relative entropy. Since 
\[
\mathbb{E}_{Q_{\delta}}\bigl[e^{(s-t)hR}\bigr]=\frac{Z(t)}{Z(s)},
\]
By Jensen's inequality, 
\[
(s-t)h\,\mathbb{E}_{Q_{\delta}}R\le\ln Z(t)-\ln Z(s).
\]
Therefore, 
\begin{align}
D(Q_{\delta}\Vert P) & =-sh\,\mathbb{E}_{Q_{\delta}}R-\ln Z(s)\nonumber \\
 & \ge-\frac{s}{s-t}\bigl(\ln Z(t)-\ln Z(s)\bigr)-\ln Z(s)\nonumber \\
 & =\frac{t}{s-t}\ln Z(s)-\frac{s}{s-t}\ln Z(t).\label{eq:D-intermediate}
\end{align}

Using $Z(s)\ge a$ and $Z(t)\le a\frac{t}{t-1},$ we obtain 
\begin{align}
D(Q_{\delta}\Vert P) & \ge\frac{t}{s-t}\ln a-\frac{s}{s-t}\left(\ln a+\ln\frac{t}{t-1}\right)\nonumber \\
 & =\ln\frac{1}{a}-\frac{s}{s-t}\ln\frac{t}{t-1}.
\end{align}
Since $\frac{s}{s-t}=\frac{1+2\delta}{\delta},\,\frac{t}{t-1}=\frac{1+\delta}{\delta},$
this is exactly 
\[
D(Q_{\delta}\Vert P)\ge\ln\frac{1}{a}-C_{\delta}.
\]
Since relative entropy is nonnegative, \eqref{eq:entropy-l} holds. 

It remains to estimate the Fisher information. The distance function
$R$ is $1$-Lipschitz and satisfies $|\nabla R|\le1\,\text{a.e.}$
Hence 
\[
\nabla\ln\frac{\d Q_{\delta}}{\d P}=-sh\nabla R\,\text{a.e.}
\]
which implies 
\begin{align}
I(Q_{\delta}\Vert P) & =\int\left|\nabla\ln\frac{\d Q_{\delta}}{\d P}\right|^{2}\d Q_{\delta}=s^{2}h^{2}\int|\nabla R|^{2}\,\d Q_{\delta}\nonumber \\
 & \le s^{2}h^{2}=(1+2\delta)^{2}h^{2}.
\end{align}
\end{proof}

We now apply Lemma~\ref{lem:LSC} to an isoperimetric minimizer and
then establish the bridge between isoperimetric and log-Sobolev inequalities. 

Let $A\subseteq\M^{n}$ be an isoperimetric minimizer of volume $a$
given in Theorem~12 of \cite{yu2025large}: 
\[
\nu_{n}(A)=a,\qquad\nu^{+}_{n}(A)=I_{n}(a).
\]
Define $h:=\frac{I_{n}(a)}{a}.$ Theorem~12 of \cite{yu2025large}
states that 
\[
\frac{1}{r}\ln\frac{\nu_{n}(A^{r})}{a}\le\frac{I_{n}(a)}{a},\;\forall r>0.
\]

Applying Lemma~\ref{lem:LSC} with $P=\nu_{n}$ and substituting
$Q_{\delta}$ into the nonlinear log-Sobolev inequality, 
\[
\frac{1}{n}I(Q_{\delta}\Vert\nu_{n})\ge\breve{\Theta}(\frac{1}{n}D(Q_{\delta}\Vert\nu_{n})).
\]
Combining this with \eqref{eq:entropy-l} and \eqref{eq:fisher-u}
yields
\[
\frac{1}{n}(1+2\delta)^{2}h^{2}\ge\breve{\Theta}\left(\frac{1}{n}(\ln\frac{1}{a}-C_{\delta})_{+}\right),
\]
i.e., \eqref{eq:theta}.  

\section{Proof of Theorem \ref{thm:main}}

We now prove Theorem \ref{thm:main}. Taking the infimum over $n$
in \eqref{eq:linear} yields 
\begin{equation}
I_{\inf}(a)\ge\frac{a}{1+2\delta}\sqrt{2K_{\mathrm{LS}}\left(\ln\frac{1}{a}-C_{\delta}\right)_{+}}.\label{eq:Iinf}
\end{equation}
Therefore 
\begin{equation}
\frac{I_{\inf}(a)}{a\sqrt{2\ln(1/a)}}\ge\frac{\sqrt{K_{\mathrm{LS}}}}{1+2\delta}\sqrt{\left(1-\frac{C_{\delta}}{\ln(1/a)}\right)_{+}}.\label{eq:ratio}
\end{equation}

For every fixed $\delta>0$, let $a\downarrow0$. Then, we obtain
\[
\liminf_{a\downarrow0}\frac{I_{\inf}(a)}{a\sqrt{2\ln(1/a)}}\ge\frac{\sqrt{K_{\mathrm{LS}}}}{1+2\delta}.
\]
Letting $\delta\downarrow0$ gives 
\begin{equation}
\liminf_{a\downarrow0}\frac{I_{\inf}(a)}{a\sqrt{2\ln(1/a)}}\ge\sqrt{K_{\mathrm{LS}}}.\label{eq:liminf}
\end{equation}
Equivalently, $K^{-}_{\mathrm{IS}}\ge K_{\mathrm{LS}}.$ 

The reverse inequality $K^{-}_{\mathrm{IS}}\le K_{\mathrm{LS}}$ was
established in Theorem~3 of \cite{yu2025large}. Consequently, $K^{-}_{\mathrm{IS}}=K_{\mathrm{LS}}.$ 

We next show that the liminf in the definition of $K^{-}_{\mathrm{IS}}$
is in fact a limit by using an argument in \cite{yu2025large}.

Let $\breve{\Theta}$ denote the optimal convex entropy--Fisher-information
profile \cite{yu2025large} so that 
\begin{equation}
\frac{1}{n}I(\mu\Vert\nu^{\otimes n})\ge\breve{\Theta}\left(\frac{D(\mu\Vert\nu^{\otimes n})}{n}\right).\label{eq:tensor}
\end{equation}

Theorem~1 of \cite{yu2025large} implies that for each fixed $\alpha>0$,
\begin{equation}
\limsup_{m\to\infty}\frac{I_{m}(e^{-m\alpha})}{e^{-m\alpha}\sqrt{m}}\le\sqrt{\breve{\Theta}(\alpha)}.\label{eq:fixed-rate}
\end{equation}

Fix $\alpha>0$. For $a\in(0,1)$, write $s:=\ln\frac{1}{a}$ and
$m:=\left\lceil \frac{s}{\alpha}\right\rceil .$ Set $b:=e^{-m\alpha}.$
Then $b\le a.$ Since $I_{m}$ is concave and $I_{m}(0)=0$, the function
$u\mapsto\frac{I_{m}(u)}{u}$ is nonincreasing on $(0,1)$. Hence
\[
\frac{I_{\inf}(a)}{a}\le\frac{I_{m}(a)}{a}\le\frac{I_{m}(b)}{b}.
\]
Thus 
\begin{align*}
\frac{I_{\inf}(a)}{a\sqrt{2s}} & \le\frac{I_{m}(e^{-m\alpha})}{e^{-m\alpha}\sqrt{m}}\sqrt{\frac{m}{2s}}.
\end{align*}
As $a\downarrow0$, we obtain $m\to\infty,\,\frac{m}{s}\to\frac{1}{\alpha}.$
Using \eqref{eq:fixed-rate}, 
\[
\limsup_{a\downarrow0}\frac{I_{\inf}(a)}{a\sqrt{2\ln(1/a)}}\le\sqrt{\frac{\breve{\Theta}(\alpha)}{2\alpha}}.
\]
Now let $\alpha\downarrow0$. By the definition of the optimal linear
log-Sobolev constant, 
\begin{equation}
\lim_{\alpha\downarrow0}\frac{\breve{\Theta}(\alpha)}{2\alpha}=K_{\mathrm{LS}}.\label{eq:slope}
\end{equation}
Therefore,
\begin{equation}
\limsup_{a\downarrow0}\frac{I_{\inf}(a)}{a\sqrt{2\ln(1/a)}}\le\sqrt{K_{\mathrm{LS}}}.\label{eq:limsup}
\end{equation}

Combining \eqref{eq:liminf} and \eqref{eq:limsup}, 
\[
\lim_{a\downarrow0}\frac{I_{\inf}(a)}{a\sqrt{2\ln(1/a)}}=\sqrt{K_{\mathrm{LS}}}.
\]

\section{Proof of Theorem \ref{thm:MDP} }

The lower bound follows immediately from Theorem~\ref{thm:finite-dimensional}:
for every fixed $\delta>0$, 
\[
\frac{I_{n}(a_{n})}{a_{n}\sqrt{2s_{n}}}\ge\frac{\sqrt{K_{\mathrm{LS}}}}{1+2\delta}\sqrt{\left(1-\frac{C_{\delta}}{s_{n}}\right)_{+}}.
\]
Since $s_{n}\to\infty$, 
\[
\liminf_{n\to\infty}\frac{I_{n}(a_{n})}{a_{n}\sqrt{2s_{n}}}\ge\frac{\sqrt{K_{\mathrm{LS}}}}{1+2\delta}.
\]
Letting $\delta\downarrow0$ gives 
\begin{equation}
\liminf_{n\to\infty}\frac{I_{n}(a_{n})}{a_{n}\sqrt{2s_{n}}}\ge\sqrt{K_{\mathrm{LS}}}.\label{eq:MD}
\end{equation}

For the upper bound, we use an argument in \cite{yu2025large}. Fix
$\alpha>0$ and define $m_{n}:=\left\lceil \frac{s_{n}}{\alpha}\right\rceil ,\,b_{n}:=e^{-m_{n}\alpha}.$
Because $\frac{s_{n}}{n}\to0,$ we have $m_{n}\le n$ for all sufficiently
large $n$.

By monotonicity of $n\mapsto I_{n}(a)$ for any $a$, $I_{n}(a_{n})\le I_{m_{n}}(a_{n}).$
Since $b_{n}\le a_{n}$ and $I_{m_{n}}(u)/u$ is nonincreasing, 
\[
\frac{I_{n}(a_{n})}{a_{n}}\le\frac{I_{m_{n}}(a_{n})}{a_{n}}\le\frac{I_{m_{n}}(b_{n})}{b_{n}}.
\]
Thus 
\[
\frac{I_{n}(a_{n})}{a_{n}\sqrt{2s_{n}}}\le\frac{I_{m_{n}}(e^{-m_{n}\alpha})}{e^{-m_{n}\alpha}\sqrt{m_{n}}}\sqrt{\frac{m_{n}}{2s_{n}}}.
\]
Since $m_{n}\to\infty,\,\frac{m_{n}}{s_{n}}\to\frac{1}{\alpha},$
the fixed-rate estimate \eqref{eq:fixed-rate} yields 
\[
\limsup_{n\to\infty}\frac{I_{n}(a_{n})}{a_{n}\sqrt{2s_{n}}}\le\sqrt{\frac{\breve{\Theta}(\alpha)}{2\alpha}}.
\]
Letting $\alpha\downarrow0$ and using \eqref{eq:slope}, 
\[
\limsup_{n\to\infty}\frac{I_{n}(a_{n})}{a_{n}\sqrt{2s_{n}}}\le\sqrt{K_{\mathrm{LS}}}.
\]
Together with \eqref{eq:MD}, this proves the result. 

\section{Proof of Theorem \ref{thm:LDP} }

The lower bound follows immediately from Theorem~\ref{thm:finite-dimensional}:
for every fixed $\delta>0$, 
\[
\frac{I_{n}(a_{n})}{a_{n}\sqrt{n}}\ge\frac{1}{1+2\delta}\sqrt{\breve{\Theta}\left((\alpha-\frac{C_{\delta}}{n})_{+}\right)},
\]
where $a_{n}=e^{-n\alpha}$. Letting $n\to\infty$ and by continuity
of $\breve{\Theta}$, 
\[
\liminf_{n\to\infty}\frac{I_{n}(a_{n})}{a_{n}\sqrt{n}}\ge\frac{1}{1+2\delta}\sqrt{\breve{\Theta}(\alpha)}.
\]
Letting $\delta\downarrow0$ gives 
\begin{equation}
\liminf_{n\to\infty}\frac{I_{n}(a_{n})}{a_{n}\sqrt{n}}\ge\sqrt{\breve{\Theta}(\alpha)}.\label{eq:MD-1}
\end{equation}

For the upper bound, we already established in \cite{yu2025large}
that 
\[
\limsup_{n\to\infty}\frac{I_{n}(a_{n})}{a_{n}\sqrt{n}}\le\sqrt{\breve{\Theta}(\alpha)}.
\]

\section{Proof of Theorem \ref{thm:approximate}}

Here we only prove \eqref{eq:CLT}. The proofs of \eqref{eq:MDP}
and \eqref{eq:LDP} are similar and hence omitted. 

Observe that given each $n$, $\{\nu^{\otimes n}_{m}\}$ converges
to $\nu^{\otimes n}$ in total-variation distance and in addition
$\nu^{\otimes n}_{m}(B)\ge(1-1/m)^{n}\nu^{\otimes n}(B)$ for any
Borel set $B$. E. Milman \cite[Theorem 6.10]{milman2009role} showed
that in this case, for all $a\in[0,1]$, $I_{\nu^{\otimes n}}(a)=\lim_{m\to\infty}I_{\nu^{\otimes n}_{m}}(a),$
and consequently, $I_{\nu^{\otimes n}}$ is concave on $[0,1]$, and
thus continuous on $(0,1)$. Let $A_{m}$ be an isoperimetric minimizer
for $\nu^{\otimes n}_{m}$ of weighted volume $a$ in the standard
sense. By Theorem~12 of \cite{yu2025large}, for any $r>0$, 
\begin{align*}
\frac{I_{\nu^{\otimes n}_{m}}(a)}{a} & \ge\frac{1}{r}\ln\frac{\nu^{\otimes n}_{m}(A^{r}_{m})}{a}\ge\frac{1}{r}\ln\frac{(1-1/m)^{n}\nu^{\otimes n}(A^{r}_{m})}{a},
\end{align*}
where $a_{m}=\nu^{\otimes n}(A_{m})$. 

Applying Lemma \ref{lem:LSC}, 
\begin{align*}
\frac{I_{\nu^{\otimes n}_{m}}(a)}{a} & \ge\frac{1}{1+2\delta}\sqrt{2K_{\mathrm{LS}}\left(\ln\frac{1}{a_{m}}-C_{\delta}\right)_{+}}+\frac{n}{r}\ln(1-\frac{1}{m}).
\end{align*}

Since $|a_{m}-a|=|\nu^{\otimes n}(A_{m})-\nu^{\otimes n}_{m}(A_{m})|\le|\nu^{\otimes n}-\nu^{\otimes n}_{m}|_{\mathrm{TV}}$,
taking the limit as $m\to\infty$ yields $a_{m}\to a$. We hence obtain
that 
\begin{align*}
\frac{I_{\nu^{\otimes n}}(a)}{a} & \ge\frac{1}{1+2\delta}\sqrt{2K_{\mathrm{LS}}\left(\ln\frac{1}{a}-C_{\delta}\right)_{+}},
\end{align*}
which implies Proposition \ref{thm:finite-dimensional} is valid in
this setting. Following the same steps as \eqref{eq:Iinf}-\eqref{eq:liminf},
we obtain 
\begin{equation}
\liminf_{a\downarrow0}\frac{I_{\inf}(a)}{a\sqrt{2\ln(1/a)}}\ge\sqrt{K_{\mathrm{LS}}}.\label{eq:liminf-lower-1}
\end{equation}

As for the other direction, the proof is identical to the smooth setting
except that Theorem~1 of \cite{yu2025large} is replaced by Theorem~2
of \cite{yu2025large}. 

\section*{Acknowledgements}

\emph{Generative-AI use disclosure:} ChatGPT 6 was used to prove Lemma~\ref{lem:LSC}.
All mathematical claims are the full responsibility of the author.

\emph{Funding:} This work was supported by the National Key Research
and Development Program of China under grant 2023YFA1009604 and the
NSFC under grant 62101286.

\bibliographystyle{abbrv}
\bibliography{ref}

\end{document}